\documentclass[12pt]{amsart}

\usepackage{mathrsfs}
\usepackage{amssymb,amsmath,amsxtra,amscd, feynmf,latexsym}
\usepackage{pst-all}
\usepackage{mathpple}
\usepackage{setspace,ifpdf}
\usepackage{diagrams}

\usepackage{pst-tree, pst-node}
\usepackage[mathscr]{eucal} 

\ifpdf
\usepackage{pdfsync}
\fi

\newcommand{\nc}{\newcommand}
\newcommand{\mc}{\mathcal}
\nc{\on}{\operatorname}
\nc{\g}{\mathfrak{g}}
\nc{\n}{\mathfrak{n}}
\nc{\ch}{\on{CH}}
\nc{\wt}{\widetilde}
\renewcommand{\P}{\mc{P}}
\nc{\F}{\mc{F}}
\nc{\C}{\mc{C}}
\nc{\M}{\on{M}}
\renewcommand{\H}{\on{H}}
\nc{\G}{\mc{G}}
\renewcommand{\S}{\mc{S}}

\newtheorem{theorem}{Theorem}
\newtheorem{lemma}{lemma}

\begin{document}

\title{Incidence categories}
\author{Matt Szczesny} \thanks{The author is supported by an NSA grant}
\address{Department of Mathematics  and Statistics, 
         Boston University, Boston MA, USA}
\email{szczesny@math.bu.edu}

\begin{abstract}
Given a family $\F$ of posets closed under disjoint unions and the operation of taking convex subposets, we construct a category $\C_{\F}$ called the \emph{incidence category of $\F$}. This category is "nearly abelian" in the sense that all morphisms have kernels/cokernels, and possesses a symmetric monoidal structure akin to direct sum. The Ringel-Hall algebra of $\C_{\F}$ is isomorphic to the incidence Hopf algebra of the collection $\P(\F)$ of order ideals of posets in $\F$. This construction generalizes the categories introduced by K. Kremnizer and the author In the case when $\F$ is the collection of posets coming from rooted forests or Feynman graphs. 
\end{abstract}

\maketitle

\section{Introduction}

The notion of the \emph{incidence algebra} of an interval-closed family $\P$ of posets was introduced by G.-C. Rota in \cite{R}. The work of W. Schmitt demonstrated that incidence algebras frequently possess important additional structure - namely that of a Hopf algebra. The seminal papers \cite{Sch2, Sch} established various key structural and combinatorial properties of  incidence Hopf algebras. 

In this paper, we show that a certain class of incidence Hopf algebras can be "categorified". Given a family $\F$ of posets which are closed under disjoint unions and the operation of taking convex subposets, we construct a category $\C_{\F}$, whose objects are in one-to-one correspondence with the posets in $\F$. While $\C_{\F}$ is not abelian (morphisms only form a set), it is "nearly" so, in the sense that all morphisms possess kernels and cokernels, it has a null object, and symmetric monoidal structure akin to direct sum. We can therefore talk about exact sequences in $\C_{\F}$, its Grothendieck group, and Yoneda Ext's. 

In particular, we can define the Ringel-Hall algebra $\H_{\C_{\F}}$ of $\C_{\F}$. $\H_{\C_{\F}}$ is the 
$\mathbb{Q}$--vector space of finitely supported functions on isomorphism classes of $\C_{\F}$:
\[
\H_{\C_{\F}} := \{ f: \on{Iso}(\C_{\F}) \rightarrow \mathbb{Q} | |supp(f)| < \infty  \}
\]
with product given by convolution:
\begin{equation} \label{prod}
f \star g (M) = \sum_{A \subset M} f(A) g(M/A). 
\end{equation}
$\H_{\C_{\F}}$ possesses a co-commutative co-product given by
\begin{equation} \label{cop}
\Delta(f)(M,N)=f(M \oplus N) 
\end{equation}
We show that $\H_{\C_{\F}}$ is isomorphic to the incidence Hopf algebra of the family $\P(\F)$ of order ideals of posets in $\F$. This is the sense in which $\C_{\F}$ is a categorification. The resulting Hopf algebra $\H_{\C_{\F}}$ is graded connected and co-commutative, and so by the Milnor-Moore theorem isomorphic to $U(\n_{\F})$, where $\n_{\F}$ is the Lie algebra of its primitive elements. 

When $\F$ is the family of posets coming from rooted forests or Feynman graphs, $\C_{\F}$ coincides with the categories introduced by K. Kremnizer and the author in \cite{KS}. 

\bigskip

\noindent {\bf Acknowledgements: } I would like to thank Muriel Livernet for valuable conversations and in particular for suggesting that there ought to be an interpretation of the categories introduced in \cite{KS} in terms of incidence algebras.  I would also like to thank Dirk Kreimer, Valerio Toledano Laredo, and Kobi Kremnizer for valuable conversations. 

\section{Recollections on posets}

We begin by recalling some basic notions and terminology pertaining to posets ( partially ordered sets) following \cite{Sch, St}. 

\begin{enumerate}
\item An \emph{interval} is a poset having unique minimal and maximal elements. For $x, y$ in a poset $P$, we denote by $[x,y]$ the interval
\[
[x,y] := \{ z \in P : x \leq z \leq y \}
\]  
If $P$ is an interval, we will often denote by $0_P$ and $1_P$ the minimal and maximal elements. 
\item An \emph{order ideal} in a poset $P$ is a subset $L \subset P$ such that whenever $y \in L$ and $x \leq y$ in $P$, then $x \in L$.
\item  A subposet $Q$ of $P$ is \emph{convex} if, whenever $x \leq y$ in $Q$ and $z \in P$ satisfies $x \leq z \leq y$, then $z \in Q$. Equivalently, $Q$ is convex if $Q = L \backslash I$ for order ideals $I \subset L$ in $P$. 
\item Given two posets $P_1, P_2$, their disjoint union is naturally a poset, which we denote by $P_1 + P_2$. In $P_1 + P_2$, $x \leq y$ if both lie in either $P_1$ or $P_2$, and $x \leq y$ there.
\item A poset which is not the union of two non-empty posets is said to be \emph{connected}. 
\item The cartesian product $P_1 \times P_2$ is a poset where $(x,y) \leq (x',y')$ iff $x \leq x'$ and $y \leq y'$. 
\item A \emph{distributive lattice} is a poset $P$ equipped with two operations $\wedge$, $\vee$ that satisfy the following properties:
\begin{enumerate}
\item $\wedge, \vee$ are commutative and associative
\item $\wedge, \vee$ are idempotent - i.e. $x \wedge x = x$, $x \vee x = x$
\item $x \wedge (x \vee y) = x = x \vee ( x \wedge y)$
\item $x \wedge y = x \iff x \vee y = y \iff x \leq y$
\item $ x \vee (y \wedge z) = (x \vee y) \wedge (x \vee z)$
\item $x \wedge (y \vee z) = (x \wedge y) \vee (x \wedge z)$
\end{enumerate}

\item For a poset $P$, denote by $J_P$ the poset of order ideals of $P$,  ordered by inclusion. $J_P$ forms a distributive lattice with $I_1 \vee I_2 := I_1 \cup I_2$ and $I_1 \wedge I_2 := I_1 \cap I_2$ for $I_1, I_2 \in J_P$. If $P_1, P_2$ are posets, we have $J_{P_1 + P_2} = J_{P_1} \times J_{P_2}$, and if $I,L \in J_P$, and $I \subset L$, then $[I,L]$ is naturally isomorphic to the lattice of order ideals $J_{L \backslash I}$. 

\end{enumerate}
\section{From posets to categories}

Let $\F$ be a family of posets which is closed under the formation of disjoint unions and the operation of taking convex subposets, and let 
$\P(\F) = \{ J_P : P \in \F \}$ be the corresponding family of distributive lattices of order ideals. For each pair 
$P_1, P_2 \in \F$, let $\M(P_1, P_2)$ denote a set of maps $P_1 \rightarrow P_2$ such that the collections 
$\M(P_1,P_2)$ satisfy the following properties:

\begin{enumerate}
\item for each $f \in \M(P_1, P_2)$, $f: P_1 \rightarrow P_2$ is a poset isomorphism
\item $M(P,P)$ contains the identity map
\item If $f \in \M(P_1, P_2)$ then $f^{-1} \in \M(P_2, P_1)$. 
\item if $f \in \M(P_1, P_2)$ and $g \in \M(P_2, P_3)$, then $g \circ f \in \M(P_1, P_3)$
\item If $f \in \M(P_1,P_2)$ and $g \in \M(Q_1, Q_2)$ then $f \cup g \in \M(P_1 + Q_1, P_2 + Q_2)$, where $f \cup g$ denotes the map induced on the disjoint union. 
\end{enumerate}
Typically, the collection $\M(P_1, P_2)$ will consist of poset isomorphisms respecting some additional structure (such as for instance a coloring). It follows from the above properties that $\M(P,P)$ forms a group, which we denote $\on{Aut}_{\M}(P)$. 

\subsection{The category $\C_{\F}$}

We proceed to define a category $\C_{\F}$, called the \emph{incidence category of $\F$} as follows. Let
$$\on{Ob}(\C_{\F} ) := \F = \{ X_P : P \in \F \}$$ and 
$$ \on{Hom}(X_{P_1}, X_{P_2}) := \{ ( I_1, I_2, f) : I_i \in J_{P_i},  f \in \M(P_1 \backslash I_1, I_2) \} \; \; i=1,2$$
We need to define the composition of morphisms
\[
\on{Hom}(X_{P_1}, X_{P_2}) \times \on{Hom}(X_{P_2}, X_{P_3}) \rightarrow \on{Hom}(X_{P_1}, X_{P_3})
\]
Suppose that $(I_1, I_2, f) \in \on{Hom}(X_{P_1}, X_{P_2})$ and $(I'_2, I'_3, g) \in \on{Hom}(X_{P_2}, X_{P_3})$. Their composition is the morphism $(K_1, K_3, h)$ defined as follows. 
\begin{itemize}
\item We have $I_2 \wedge I'_2 \subset I_2$, and since $f: P_1 \backslash I_1 \rightarrow I_2$ is an isomorphism, $f^{-1}(I_2 \wedge I'_2)$ is an order ideal of $P_1 \backslash I_1$. Since in $J_{P_1}$, $[I_1, P] \simeq J_{P_1 \backslash I_1}$, we have that $f^{-1}(I_2 \wedge I'_2)$ corresponds to an order ideal $K_1 \in J_{P_1}$ such that $I_1 \subset K_1$. 
\item We have $I'_2 \subset I_2 \vee I'_2$, and since $ [I'_2, P_2] \simeq J_{P_2 \backslash I'_2}$, $I_2 \vee I'_2$ corresponds to an order ideal $L_2 \in J_{P_2 \backslash I'_2}$. Since $g: P_2 \backslash I'_2 \rightarrow I'_3$ is an isomorphism, $g(L_2) \subset J_{I'_3}$, and since $J_{I'_3} \subset J_{P_3}$, $g(L_2)$ corresponds to an order ideal $K_3 \in J_{P_3}$ contained in $I'_3$. 
\item  The isomorphism $f: P_1 \backslash I_1 \rightarrow I_2$ restricts to an isomorphism $\bar{f}: P_1 \backslash K_1 \rightarrow I_2 \backslash I_2 \wedge I'_2 = I_2 \backslash I'_2 $, and the isomorphism $g: P_2 \backslash I'_2$ restricts to an isomorphism $\bar{g}: I_2 \vee I'_2 \backslash I'_2 = I_2 \backslash I'_2 \rightarrow  K_3$. Thus, $g \circ f: P_1 \backslash K_1 \rightarrow K_3$ is an isomorphism and $g \circ f \in \M(P_1 \backslash K_1, K_3)$ by the property $(4)$ above. 
 \end{itemize}

\begin{lemma}
Composition of morphisms is associative. 
\end{lemma}
\begin{proof}
 Suppose that $P_1, P_2, P_3, P_4 \in \F$, and that we have three morphisms as follows:
\[
X_{P_1} \overset{(A_1,A_2,f )}{\longrightarrow} X_{P_2} \overset{(B_2,B_3,g)}{\longrightarrow}   X_{P_3} \overset{(C_3,C_4,h )}{\longrightarrow} X_{P_4}
\]
Given a poset $P$ and subsets $S_1, \cdots, S_k$ of $P$, denote by $[S_1, \cdots, S_k]$ the smallest order ideal containing $\cup^{k}_{i=1}S_i $. 
We have :
\begin{align*}
(C_3,C_4,h) \circ  &((B_2,B_3,g) \circ (A_1,A_2,f)) \\ & = (C_3,C_4,h) \circ ([A_1,f^{-1}(B_2 )], g(A_2 \backslash B_2), g \circ f) \\
&= ([A_1, f^{-1}(B_2),  (g \circ f)^{-1} (C_3)], h(g(A_2 \backslash B_2) \backslash C_3), h \circ g \circ f)
\end{align*}
whereas
\begin{align*}
((C_3,C_4,h) \circ & (B_2,B_3,g) ) \circ (A_1,A_2,f) \\
&= ([B_2, g^{-1}(C_3)], h(B_3 \backslash C_3), h \circ g)  \circ (A_1,A_2,f) \\
&= ([A_1, f^{-1}([B_2, g^{-1}(C_3)])], h \circ g (A_2 \backslash [B_2, g^{-1}(C_3)]), h \circ g \circ f )
\end{align*}
We have $$f^{-1}([B_2, g^{-1}(C_3)]) = f^{-1}(B_2 \cup g^{-1}(C_3)) = f^{-1}(B_2) \cup (g \circ f)^{-1} (C_3)$$ which implies that
$$ [A_1, f^{-1}(B_2),  (g \circ f)^{-1} (C_3)] = [A_1, f^{-1}([B_2, g^{-1}(C_3)])] $$ and
$$ h(g(A_2 \backslash B_2) \backslash C_3) = h \circ g ( A_2 \backslash (B_2 \cup g^{-1} (C_3) )) =  h \circ g (A_2 \backslash [B_2, g^{-1}(C_3)]) $$
This proves the two compositions are equal. 
\end{proof}
\bigskip

\noindent Finally,
\begin{itemize}
\item We refer to $X_{I_2}$ as the \emph{image} of the morphism $(I_1, I_2, f): X_{P_1} \rightarrow X_{P_2}$. 
\item We denote by $\on{Iso}(\C_{\F})$ the collection of isomorphism classes of objects in $\C_{\F}$, and by $[X_P]$ the isomorphism class of $X_P \in \C_{\F}$. 
\end{itemize}

\section{Properties of the categories $\C_{\F}$}

We now enumerate some of the properties of the categories $\C_{\F}$. 
\bigskip
\begin{enumerate}
\item The empty poset $\emptyset$ is an initial, terminal, and therefore null object. We will sometimes denote it by $X_{\emptyset}$.   
\bigskip
\item We can equip $\C_{\F}$ with a symmetric monoidal structure by defining
\[
X_{P_1} \oplus X_{P_2}  := X_{P_1 + P_2}.
\]
\item The indecomposable objects of $\C_{\F}$ are the $X_P$ with $P$ a connected poset in $\F$.
\bigskip
\item The irreducible objects of $\C_{\F}$ are the $X_{P}$ where $P$ is a one-element poset. 
\bigskip
\item  \label{kernel} Every morphism 
\begin{equation} \label{morphism}
(I_1, I_2,f) : X_{P_1} \rightarrow X_{P_2}
\end{equation}
has a kernel
\[
(\emptyset,I_1, id):  X_{I_1} \rightarrow X_{P_1} 
\]
\bigskip
\item \label{cokernel} Similarly, every morphism \ref{morphism} possesses a cokernel
\[
(I_2, P_2 \backslash I_2, id) : X_{P_2} \rightarrow X_{P_2 \backslash I_2}
\]
\medskip
\noindent We will frequently use the notation $X_{P_2}/X_{P_1}$ for $coker((I_1,I_2,f))$. 

\bigskip
\noindent {\bf Note:} Properties \ref{kernel} and \ref{cokernel} imply that the notion of exact sequence makes sense in $\C_{\F}$.
\bigskip
\item All monomorphisms are of the form
\[
(\emptyset,I, f) : X_Q \rightarrow X_{P}
\]
where $I \in J_P$, and $f: Q \rightarrow I \in \M(Q,I)$. Monomorphisms $X_Q \rightarrow X_P$ with a fixed image $X_{I}$ form a torsor over $\on{Aut}_{\M}(I)$. 
All epimorphisms are of the form
\[
(I,\emptyset, g) : X_P \rightarrow X_{Q}
\]
where $I \in J_P$ and $g: P \backslash I \rightarrow Q \in \M(P \backslash I, Q)$. Epimorphisms with fixed kernel $X_I$ form a torsor over $\on{Aut}_{\M}(P \backslash I)$ 
\bigskip
\item
Sequences of the form \label{propses}
\begin{equation} \label{ses}
X_{\emptyset} \overset{(\emptyset, \emptyset, id)}{\rightarrow} X_{I} \overset{(\emptyset,I, id)}{\longrightarrow} X_{P} \overset {(I,\emptyset,id)}{\longrightarrow} X_{P \backslash I} \overset{(P \backslash I, \emptyset, id)}{\rightarrow} X_{\emptyset}
\end{equation} 
with $I \in J_P$ are short exact, and all other short exact sequences with $X_P$ in the middle arise by composing with isomorphisms $X_I \rightarrow X_{I'}$ and $X_{P \backslash I} \rightarrow X_Q$ on the left and right. 
\bigskip
\item

 \label{quotients}  Given an object $X_{P}$ and a subobject $X_{I}, I \in J_P$,  the isomorphism $J_{P \backslash I} \simeq [I,P]$ translates into the statement that there is a bijection between subobjects of $X_P/X_I$ and order ideals $J \in J_P$ such that $I \subset J \subset P$ via $X_J \leftrightarrow J$. The bijection is compatible with quotients, in the sense that $(X_P/X_I)/(X_J/X_I) \simeq X_J/X_I$. 
 

\bigskip
\item Since the posets in $\F$ are finite, $\on{Hom}(X_{P_1},X_{P_2})$ is a finite set. 
\bigskip
\item We may define Yoneda $\on{Ext}^{n}(X_{P_1}, X_{P_2})$ as the equivalence class of $n$--step exact sequences with $X_{P_1}, X_{P_2}$ on the right and left respectively.  $\on{Ext}^{n} (X_{P_1}, X_{P_2})$ is a finite set. Concatenation of exact sequences makes $$\mathbb{E}xt^* := \cup_{A,B \in I(\C_{\F}), n} \on{Ext}^n (A,B)$$ into a monoid.  
\bigskip
\item We may define the Grothendieck group of $\C_{\F}$, $K_0(\C_{\F})$, as 
\[
K(\C_{\F}) = \bigoplus_{A \in \C_{\F}} \mathbb{Z}[A] / \sim
\]
where $\sim$ is generated by  $A+B-C$ for short exact sequences
\[
X_{\emptyset} \rightarrow A \rightarrow C \rightarrow B \rightarrow X_{\emptyset}
\]
We denote by $k(A)$ the class of an object in $K_0(\C_{\F})$. 
\end{enumerate}

\section{Ringel-Hall algebras and incidence algebras}

\subsection{Incidence Hopf algebras}

We begin by recalling the definition of the incidence Hopf algebra of a hereditary interval-closed family of posets introduced in \cite{Sch}. Incidence algebras of posets were originally introduced in \cite{R}. 

A family $\P$ of finite intervals is said to be \emph{interval closed}, if it is non-empty, and for all $P \in \P$ and $x \leq y \in P$, the interval $[x,y]$ belongs to $\P$. An \emph{order compatible relation} on an interval closed family $\P$ is an equivalence relation $\sim$ such that whenever $P \sim Q$ in $\P$, there exists a bijection $\phi: P \rightarrow Q$ such that $[0_P, x] \sim [0_Q,\phi(x)]$ and $[x,1_P] \sim [\phi(x), 1_Q]$ for all $x \in P$. Typical examples of order compatible relations are poset isomorphism, or isomorphism preserving some additional structure (such as a coloring). We denote by $\tilde{\P}$ the set equivalence classes of $\P$ under $\sim$, i.e. $\tilde{\P} = \P/\sim$, and by $[P]$ the equivalence class of a poset $P$ in $\tilde{\P}$.   The \emph{incidence algebra} of the family $(\P, \sim)$, denoted $\H_{\P,\sim}$, is
\[
\H_{\P,\sim} := \{ f: \tilde{\P} \rightarrow \mathbb{Q} : |supp(f)| < \infty \}
\] 
(note that the finiteness of the support of $f$ is not standard). $\H_{\P,\sim}$ is naturally a $\mathbb{Q}$--vector space, and becomes an associative $\mathbb{Q}$--algebra under the convolution product
\begin{equation} \label{the_prod}
f \bullet g ([P]) := \sum_{x \in P} f ([0_P,x]) g([x,1_P])
\end{equation}

We would now like to equip $\H_{\P, \sim}$ with a Hopf algebra structure. To do this, the family $\P$ and the relation $\sim$ must satisfy some additional properties. A \emph{hereditary family} is an interval closed family  $\P$ of posets which is closed under the formation of direct products. We will assume that $\sim$ satisfies the following two properties:
\begin{itemize}
\item whenever $P \sim Q$ in $\P$, then $P \times R \sim Q \times R$ and $R \times P \sim R \times Q$ for all $R \in \P$
\item if $P,Q \in \P$ and $|Q|=1$, then $P \times Q \sim Q \times P \sim P$
\end{itemize}
An order compatible relation $\sim$ on $\P$ satisfying these additional properties is called a \emph{Hopf relation}. 
In this case, $\tilde{\P}$ is a monoid under the operation $[P][Q] = [P \times Q]$ with unit the class of any one-element poset. 
We may now introduce a coproduct on $\H_{\P,\sim}$ 
\begin{equation} \label{the_coprod}
\Delta(f)([M], [N]) := f([M \times N])
\end{equation}
It is shown in \cite{Sch} that $\Delta$ equips $\H_{\P,\sim}$ with the structure of a bialgebra, and furthermore, that an antipode exists, making it into a Hopf algebra, which we call the \emph{incidence Hopf algebra} of $(\P,\sim)$. 

\bigskip

\noindent { \bf Note: } the original definition of incidence Hopf algebra given in \cite{Sch} is dual to the one used here.  

\bigskip

Suppose that $\F$ is a collection of finite posets which is closed under the operation of taking convex subposets and disjoint unions. The collection of order ideals $$\P(\F) := \{ J_P : P \in \F \}$$ is then a collection of intervals which is interval-closed and hereditary. If we define $J_P \sim J_Q$ whenever there exists an $f: P \rightarrow Q \in \M(P,Q)$, then $\sim$ is a Hopf relation, and so we may consider the Hopf algebra $\H_{\P(\F), \sim}$. Using the fact that for $I, L \in J_P$, $[I,L] \simeq J_{L \backslash I }$ and $J_{P + Q} \simeq J_P \times J_Q$, the product \ref{the_prod} and coproduct \ref{the_coprod} become:
\begin{equation} \label{s_prod}
f \bullet g ([J_P]) := \sum_{I \in J_P} f([J_I])g([J_{P \backslash I}])
\end{equation}
\begin{equation} \label{s_coprod}
\Delta(f)([J_P], [J_Q]) := f([J_{P  + Q}])
\end{equation}

\subsection{Ringel-Hall algebras }

 For an introduction to Ringel-Hall algebras in the context of abelian categories, see \cite{S}. 
We define the Ringel-Hall algebra of $\C_{\F}$, denoted $\H_{\C_{\F}}$, to be the 
$\mathbb{Q}$--vector space of finitely supported functions on isomorphism classes of $\C_{\F}$. I.e.
\[
\H_{\C_{\F}} := \{ f: \on{Iso}(\C_{\F}) \rightarrow \mathbb{Q} | |supp(f)| < \infty  \}
\]
As a $\mathbb{Q}$--vector space it is spanned by the delta functions $\delta_A, A \in \on{Iso}(\C_{\F})$. The algebra structure on $\H_{\C_{\F}}$ is given by the convolution product:
\begin{equation} \label{prod}
f \star g (M) = \sum_{A \subset M} f(A) g(M/A) 
\end{equation}
$\H_{\C_{\F}}$ possesses a co-commutative co-product given by
\begin{equation} \label{cop}
\Delta(f)(M,N)=f(M \oplus N) 
\end{equation}
as well as a natural $K^+_0 (\C_{\F})$--grading in which  $\delta_A$ has degree $k(A) \in K^+_0 (\C_{\F})$. 

The subobjects of  $X_P \in \C_{\F}$ are exactly $X_I$ for $I \in J_P$, and the product \ref{prod} becomes
\[
f \star g ([X_P]) = \sum_{I \in J_P} f([X_I]) g([X_{P \backslash I}])
\]
while the coproduct becomes
\[
\Delta(f)([X_P], [X_Q]) = f([X_{P + Q}])
\]
Thus, the map $$ \phi: \H_{\C_{\F}} \rightarrow \H_{\P(\F), \sim} $$
 determined by 
 $$ \phi(f)([J_P]) := f([X_P]) $$
 is an isomorphism of Hopf algebras. Recall that a Hopf algebra $A$ over a field $k$ is \emph{connected} if it possesses a $\mathbb{Z}_{\geq 0}$--grading such that $A_0 = k $. In addition to the $K^+_0(\C_{\F})$--grading, $\H_{\C_{\F}}$ possesses a grading by the order of the poset - i.e. we may assign $\deg(\delta_{X_P}) = |P|$. This gives it the structure of graded connected Hopf algebra.  The Milnor-Moore theorem implies that $\H_{\C_{\F}}$ is the enveloping algebra of the Lie algebra of its primitive elements, which we denote by $\n_{\F}$ - i.e. $\H_{\C_{F}} \simeq U(\n_{\F})$. We have thus established the following:
 
 \begin{theorem}
 The Ringel-Hall algebra of the category $\C_{\F}$ is isomorphic to the Incidence Hopf algebra of the family $\P(\F)$. These are graded connected Hopf algebras, graded by the order of poset, and isomorphic to $U(\n_{\F})$, where $\n_{\F}$ denotes the Lie algebra of its primitive elements. 
 \end{theorem}


\section{Examples}

In this section, we give some examples of families $\F$ of posets closed under disjoint unions and the operation of taking convex subposets. 

\bigskip

\noindent {\bf Example 1: }
Let $\F= \on{Fin}$ be the collection of all finite posets, and let $M(P,Q)$ consist of poset isomorphisms. $\on{Fin}$ is clearly closed under disjoint unions and taking convex subposets.  
I claim that $K_0 (\C_{\on{Fin}}) = \mathbb{Z}$. Let $P$ be a finite poset, and $m 
\in \P$ a minimal element. Then $m$ is also an order ideal in $P$, and so we have an exact sequence
\[
\emptyset \rightarrow X_{\bullet} \rightarrow X_{P} \rightarrow X_{P \backslash m} \rightarrow \emptyset
\]
where $\bullet$ denotes the one-element poset. Repeating this procedure with $P \backslash m$, we see that in $K_0 (\C_{\on{Fin}})$ every element can be written as multiple of $X_{\bullet}$, with $X_{P} \sim |P| X_{\bullet}$, and $K_0 (\C_{\on{Fin}}) \simeq \mathbb{Z}$. Thus, the $K_0^+$--grading on $\H_{\C_{\on{Fin}}}$ coincides with that by order of poset. 

$\on{Iso}(\C_{\on{Fin}})$ consists of isomorphism classes of finite posets. 
The Lie algebra $\n_{\on{Fin}}$ is spanned by $\delta_{[X_P]}$ for $P$ connected posets. For $P,Q \in \on{Fin}$, both connected, we have
\[
\delta_{[X_P]} \star \delta_{[X_Q]} = \sum_{ \{ \substack{ R \in \on{Iso}(\on{Fin}) | P \in J(R), \\ Q \simeq R \backslash P}\}}  N(P,Q;R) \delta_{[X_R]}
\]
where $N(P,Q;R):=|\{I \in J_R| I \simeq P\}|$, and
\[
[\delta_{[X_P]}, \delta_{[X_Q]}] = \delta_{[X_P]} \star \delta_{[X_Q]} - \delta_{[X_Q]} \star \delta_{[X_P]}.
\]

\bigskip

\noindent {\bf Example 2: } 
Let $\F=\S$ denote the collection of all finite sets (including the empty set). A finite set can be viewed as a poset where any two distinct elements are incomparable. $\S$ is clearly closed under disjoint unions and taking convex subposets (which coincide with subsets). Also, the order ideals of $S \in \S$ are exactly the  subsets of $S$. Let $M(P,Q)$ consist of set isomorphisms.  By the same argument as in the previous example, we have $K_0 (\C_{\S}) \simeq \mathbb{Z}$. $\on{Iso}(\C_{S}) \simeq \mathbb{Z}_{\geq 0}$, and we denote by $[n]$ the isomorphism class of the set with $n$ elements.  We have
\[
\delta_{[n]} \star \delta_{[m]} = {{m+n} \choose n} \delta_{[m+n]}
\]
$\n_{S}$ is abelian, and isomorphic the one-dimensional Lie algebra spanned by $\delta_{[1]}$. $\H_{\C_{\S}}$ is dual to the binomial Hopf algebra in \cite{Sch}. 

\bigskip

\noindent {\bf Example 3:}

Let $\S\{k\}$ denote the collection of all finite $k$--colored sets (including the empty set), and take $M(P,Q)$ to consist of color-preserving set isomorphism. Clearly, the previous example corresponds to $k=1$. We have $K_0 (\C_{\S\{ k\}}) \simeq \mathbb{Z}^k$ and  $\on{Iso}(\C_{\S\{k\}}) \simeq (\mathbb{Z}_{\geq 0})^k$. Denoting by $[n_1, \cdots, n_k]$ the isomorphism class of the set consisting of $n_1$ elements of color $c_1$, $\cdots$, $n_k$ elements of color $c_k$, we have
\[
\delta_{[n_1, \cdots, n_k]} \star \delta_{[m_1, \cdots, m_k]} = \left( \prod^{k}_{i} {{n_i+m_i} \choose n_i} \right) \delta_{[n_1+m_1, \cdots, n_k + m_k]}
\]
$\n_{\S\{ k\}}$ is a $k$--dimensional abelian Lie algebra spanned by $\delta_{e_i}$, where $e_i$ denotes the $k$--tuple $[0, \cdots ,1, \cdots, 0]$ with $1$ in the $i$th spot. 

\bigskip

The following two examples are treated in detail in \cite{KS}, and we refer the reader there for details. Please see also \cite{Sch,FG}. The resulting Hopf algebras (or rather their duals) introduced in \cite{K, CK}, form the algebraic backbone of the renormalization process in quantum field theory. The corresponding Lie algebras $\n_{\F}$ are studied in \cite{CK2}. 

\bigskip

\noindent {\bf Example 4:}

Recall that a rooted tree $t$ defines a poset whose Hasse diagram is $t$. 
Let $\F = \on{RF}$ denote the family of posets defined by rooted forests ( i.e. disjoint unions of rooted trees). It is obviously closed under disjoint unions. An order ideal in a rooted forest corresponds to an \emph{admissible cut} in the sense of \cite{CK} - that is, a cut having the property that any path from root to leaf encounters at most one cut edge, and so $\on{RF}$ is closed under the operation of taking convex posets. 
We have $K_0 (\C_{\on{RF}}) = \mathbb{Z}$, and as shown in \cite{KS},  $\H_{\C_{\on{RF}}}$ is isomorphic to the dual of the Connes-Kreimer Hopf algebra on forests, or equivalently, to the Grossman-Larson Hopf algebra. 

This example also has a colored version, where we consider the family $\on{RF}\{k \}$ of $k$--colored rooted forests. In this case, $K_0(\C_{\on{RF}\{k \}}) \simeq (\mathbb{Z}_{\geq 0})^k$.

\bigskip

\noindent {\bf Example 5:}

A graph determines a poset of subgraphs under inclusion. We may for instance consider the poset $\F = \on{FG}$ of subgraphs of Feynman graphs of a quantum field theory such as $\phi^3$ theory, which is closed under disjoint unions and convex posets. The Ringel-Hall algebra $\H_{\C_{\on{FG}}}$ is isomorphic to the dual of the Connes-Kreimer Hopf algebra on Feynman graphs. Please see \cite{KS} for details. 

In the case of $\phi^3$ theory, it is shown in \cite{KrS} that $K_0 (\C_{\on{FG}}) \simeq \mathbb{Z}[p], p \in \mathfrak{P}$, where $\mathfrak{P}$ is the set of primitively divergent graphs.


\newpage

\begin{thebibliography}{99}

\bibitem{C} Cartier, P.  A primer of Hopf algebras.  Frontiers in number theory, physics, and geometry. II,  537--615, Springer, Berlin, 2007.

\bibitem{CK} Connes, A. and Kreimer, D. Hopf algebras, renormalization, and noncommutative geometry. \emph{ Comm. Math. Phys. } {\bf 199} 203-242 (1998). 

\bibitem{CK2} Connes, A.; Kreimer, D. Insertion and elimination: the doubly infinite Lie algebra of Feynman graphs.   \emph{Ann. Henri PoincarŽ}  {\bf 3}  no. 3, 411--433 (2002).

\bibitem{FG} Figueroa H. and Gracia-Bondia J. M.  Combinatorial Hopf algebras in quantum field theory I. \emph{Rev.Math.Phys.} {\bf 17} (2005) 881.

 \bibitem{J} Joyce, D. Configurations in abelian categories. II. Ringel-Hall algebras. \emph{ Adv. Math. } {\bf 210} no. 2, 635--706 (2007).

\bibitem{K} Kreimer, D. On the Hopf algebra structure of perturbative quantum field theory. \emph{Adv. Theor. Math. Phys.} {\bf 2} 303-334 (1998).

\bibitem{KrS} Kreimer D. and Szczesny M. A Jordan-Holder theorem for Feynman graphs and noncommutative symmetric functions. In preparation. 

\bibitem{KS} Kremnizer K. and Szczesny M. Feynman graphs, rooted trees, and Ringel-Hall algebras. \emph{Comm. Math. Phys.} {\bf 289} (2009), no. 2 561--577.

\bibitem{R} Rota, G-C.  On the Foundations of Combinatorial Theory I: Theory of Mšbius Functions", Zeitschrift fŸr Wahrscheinlichkeitstheorie und Verwandte Gebiete 2: 340Ð368, (1964). 

\bibitem{S} Schiffmann, O. Lectures on Hall algebras. Preprint math.RT/0611617. 

\bibitem{Sch2} Schmitt, W. R. Antipodes and Incidence Coalgebras. \emph{Journal of Comb. Theory. A} {\bf 46} (1987), 264-290. 

\bibitem{Sch} Schmitt, W. R. Incidence Hopf algebras. \emph{  J. Pure Appl. Algebra}  {\bf 96}  (1994),  no. 3, 299--330. 

\bibitem{St} Stanley, R. Enumerative combinatorics Vol. 1. Cambridge Studies in Advanced Mathematics, 49. Cambridge University Press, Cambridge, 1997.

\end{thebibliography}
\end{document}